\documentclass[11pt]{article}
\usepackage[margin=1in]{geometry}
\usepackage{amsmath,amssymb,amsthm}
\usepackage{booktabs}

\newtheorem{theorem}{Theorem}
\newtheorem{lemma}[theorem]{Lemma}

\newtheorem{corollary}[theorem]{Corollary}
\newtheorem{remark}[theorem]{Remark}
\newtheorem{definition}[theorem]{Definition}

\title{Asymptotic Tightness of the Pigeonhole Bound\\for Large-Order Davenport--Schinzel Sequences}
\author{Jesse Geneson}
\date{}

\begin{document}
\maketitle

\begin{abstract}
We prove that the pigeonhole upper bound $\lambda(s,m) \leq \binom{m}{2}(s+1)$ is asymptotically tight whenever $s/\!\sqrt{m} \to \infty$.  In particular, $\lambda(s,m) \sim \binom{m}{2}\,s$ in this regime.  As corollaries: $\lambda(n,n)/n^3 \to \frac{1}{2}$, resolving the leading constant from the previously known interval $[\frac{1}{3}, \frac{1}{2}]$; and more generally $\lambda(an,bn) \sim \frac{ab^2}{2}\,n^3$ for any constants $a,b > 0$.  
\end{abstract}

\section{Introduction}

A \emph{Davenport--Schinzel sequence} of order~$s$ over an $m$-letter alphabet is a sequence with no immediate repetitions and no alternating subsequence of length $s + 2$.  Let $\lambda(s,m)$ denote the maximum length of such a sequence.  A standard pigeonhole argument~\cite[p.\,3]{K02} gives
\begin{equation}\label{eq:ub}
\lambda(s,m) \;\leq\; \binom{m}{2}(s+1).
\end{equation}
When $s$ is fixed, this is far from tight: $\lambda(s,m)$ is only slightly superlinear in~$m$~\cite{ASS89,N10,P15}.  However, Roselle and Stanton~\cite{RS71} showed that for fixed~$m$, $\lim_{s \to \infty} \lambda(s,m)/s = \binom{m}{2}$, so the bound is sharp in the leading constant as a function of~$s$ alone.

When both $s$ and $m$ grow, the situation is less understood.  Wellman and Pettie~\cite{WP16} introduced new constructions bridging the fixed-order and large-order regimes, but noted that on the diagonal $s = m$, the leading constant of $\lambda(m,m)/m^3$ was known only to lie in $[\frac{1}{3}, \frac{1}{2}]$.

We prove that the $S_2$ construction of~\cite{WP16}, with parameters chosen to maximize length subject to the order constraint, already matches~\eqref{eq:ub} in the leading term whenever $s$ grows faster than $\sqrt{m}$.

\section{A Uniform Lower Bound via the Roselle--Stanton Construction}

We recall the Roselle--Stanton construction~\cite{RS71} as presented in~\cite[Section~1.2]{WP16}.

\begin{definition}\label{def:rs}
For integers $s \geq 2$ and $m \geq 2$, the sequence $\mathrm{RS}(s,m)$ over the alphabet $[m] = \{1,\ldots,m\}$ is defined recursively as follows.

\medskip\noindent\textbf{Base cases.}
\begin{itemize}
\item $\mathrm{RS}(s,2) = 1,2,1,2,\ldots$ of length $s + 1$ (alternating, starting with~$1$).
\item $\mathrm{RS}(2,m) = 1,2,1,3,1,4,\ldots,1,m,1$ of length $2m-1$ for $m \geq 3$.
\end{itemize}

\medskip\noindent\textbf{Recursive case} ($s \geq 3$, $m \geq 3$):
\[
\mathrm{RS}(s,m) \;=\; \mathrm{Alt}(s,m) \;\cdot\; \mathrm{RS}(s-1,\,m-1)[m,\,m\!-\!1,\,\ldots,\,2],
\]
where $\mathrm{RS}(s-1,m-1)[m,m\!-\!1,\ldots,2]$ denotes a copy of $\mathrm{RS}(s-1,m-1)$ with symbol~$k$ relabeled to symbol~$m - k + 1$ (so first appearances occur in the order $m, m\!-\!1, \ldots, 2$), and
\begin{equation}\label{eq:alt}
\mathrm{Alt}(s,m) \;=\; \underbrace{1,2,\,1,2,\,\ldots,\,1,2}_{\lceil(s-2)/2\rceil\text{ pairs}} ,\; \underbrace{1,3,\,1,3,\,\ldots,\,1,3}_{\lceil(s-2)/2\rceil\text{ pairs}} ,\;\ldots\;,\; \underbrace{1,m,\,1,m,\,\ldots,\,1,m}_{\lceil(s-2)/2\rceil\text{ pairs}} ,\; 1.
\end{equation}
That is, $\mathrm{Alt}(s,m)$ consists of, for each $k = 2, \ldots, m$ in succession, exactly $\lceil(s-2)/2\rceil$ copies of the pair $1,k$, followed by a single trailing~$1$.
\end{definition}

\noindent The reversal of the alphabet in the recursive call is essential: since $\mathrm{Alt}(s,m)$ introduces the symbols $2, 3, \ldots, m$ in increasing order, the recursive part must introduce them in decreasing order to avoid unnecessary alternations between non-retired symbols~\cite[Section~1.2]{WP16}.

It is straightforward to verify that $\mathrm{RS}(s,m)$ is a DS$(s,m)$ sequence (see~\cite{RS71} or~\cite[Section~1.2]{WP16}), so $\lambda(s,m) \geq |\mathrm{RS}(s,m)|$.

\begin{lemma}\label{lem:altlen}
$|\mathrm{Alt}(s,m)| = 2(m-1)\lceil(s-2)/2\rceil + 1$.
\end{lemma}

\begin{proof}
There are $m - 1$ blocks (one for each $k \in \{2,\ldots,m\}$), each of length $2\lceil(s-2)/2\rceil$, plus the trailing~$1$.
\end{proof}

\noindent Denoting $L(s,m) = |\mathrm{RS}(s,m)|$, the length satisfies the base cases
\[
L(s,2) = s+1, \qquad L(2,m) = 2m-1,
\]
and, for $s \geq 3$ and $m \geq 3$, the recurrence
\begin{equation}\label{eq:rsrec}
L(s,m) \;=\; 2(m-1)\bigl\lceil(s-2)/2\bigr\rceil + 1 \;+\; L(s-1,\,m-1).
\end{equation}

\subsection{A Uniform Lower Bound}

The classical consequence of~\eqref{eq:rsrec} is that $\lambda(s,m)/s \to \binom{m}{2}$ as $s \to \infty$ for fixed~$m$.  We need a version valid when $m$ grows with~$s$.

\begin{lemma}\label{lem:rs}
For all integers $s \geq m \geq 2$,
\[
\lambda(s,m) \;\geq\; L(s,m) \;\geq\; \binom{m}{2}(s - m) + m.
\]
\end{lemma}

\begin{proof}
We prove $L(s,m) \geq \binom{m}{2}(s-m) + m$ by induction on~$m$.

\medskip\noindent\textbf{Base case} ($m = 2$): $L(s,2) = s + 1 \geq \binom{2}{2}(s - 2) + 2 = s$.

\medskip\noindent\textbf{Inductive step} ($m \geq 3$): By~\eqref{eq:rsrec} and the inductive hypothesis (applicable since $s - 1 \geq m - 1$):
\begin{align*}
L(s,m) &= 2(m\!-\!1)\bigl\lceil(s\!-\!2)/2\bigr\rceil + 1 + L(s\!-\!1,\,m\!-\!1) \\[3pt]
&\geq 2(m\!-\!1)\bigl\lceil(s\!-\!2)/2\bigr\rceil + 1 + \binom{m-1}{2}(s - m) + (m\!-\!1) \\[3pt]
&= 2(m\!-\!1)\bigl\lceil(s\!-\!2)/2\bigr\rceil + \binom{m-1}{2}(s-m) + m.
\end{align*}
It suffices to show $2(m\!-\!1)\lceil(s\!-\!2)/2\rceil \geq (m\!-\!1)(s - m)$, since $\binom{m-1}{2} + (m\!-\!1) = \binom{m}{2}$.  Since $2\lceil k/2 \rceil \geq k$, we have $2\lceil(s-2)/2\rceil \geq s - 2 \geq s - m$.
\end{proof}

\section{The $S_2$ Construction}

Let $q$ be a prime power. Following~\cite[Theorem~3.1]{WP16}, construct $A \in \{0,1\}^{q^2 \times q^2}$ with rows $(x,v) \in \mathbb{F}_q^2$, columns $(c_0, c_1) \in \mathbb{F}_q^2$, and $A((x,v),(c_0,c_1)) = 1$ iff $c_0 + c_1 x = v$.  This matrix has $q$~ones per row and column, and avoids $2 \times 2$ all-ones submatrices.

\subsection{Sequence}
For a parameter $\hat{s} \geq q$, let $C_i$ be the column support of row~$i$ ($|C_i| = q$).  For each~$i$, fix a DS$(\hat{s}, q)$ sequence~$\sigma_i$ of length $\lambda(\hat{s},q)$ over the alphabet~$C_i$. Define
\[
S_2(\hat{s}, q) \;=\; \sigma_1 \cdot \sigma_2 \cdots \sigma_{q^2},
\]
removing any immediate repetitions at copy boundaries.  This deletion affects at most $q^2 - 1$ symbols and does not increase alternation lengths.  The construction is exactly the specialization $t = 2$ of $S_t(\hat{s}, q)$ defined in~\cite[Section~4]{WP16}.

\subsection{Properties}
By~\cite[Section~4]{WP16} specialized to $t = 2$:
\begin{itemize}
\item \textbf{Alphabet size}: $q^2$ (the columns of~$A$).
\item \textbf{Order}: at most $\hat{s} + 2q - 2$.  This is the $t = 2$ case of the recursion for the maximum alternation length~$s_t$ in~$S_t$ given in~\cite[Section~4]{WP16}, namely $s_1 = \hat{s} + 1$ and $s_t \leq (t-1)\,s_{t-1} + 2(q - t + 1)$; hence $s_2 \leq \hat{s} + 2q - 1$ and the DS order is $\leq \hat{s} + 2q - 2$.
\item \textbf{Length}: $q^2 \cdot \lambda(\hat{s},q) - O(q^2) \;\geq\; q^2\bigl[\binom{q}{2}(\hat{s} - q) + q\bigr] - O(q^2)$ by Lemma~\ref{lem:rs}.
\end{itemize}

\section{Main Result}

\begin{theorem}\label{thm:general}
Let $s = s(n)$ and $m = m(n)$ be functions tending to infinity with $s/\!\sqrt{m} \to \infty$.  Then
\[
\lambda(s, m) \;=\; \binom{m}{2}\,s\;\bigl(1 - o(1)\bigr).
\]
\end{theorem}

\begin{proof}
The upper bound $\lambda(s,m) \leq \binom{m}{2}(s+1) = \binom{m}{2}s\,(1 + 1/s)$ is immediate from~\eqref{eq:ub}.

For the lower bound, let $q$ be the largest prime with $q \le \sqrt{m}$.  
By the prime number theorem, $q = \sqrt{m}\,(1 - o(1))$, hence
$q^2 = m\,(1 - o(1))$.  Set
\[
\hat{s} \;=\; s - 2q + 2.
\]
Since $s/\!\sqrt{m} \to \infty$ and $q \le \sqrt{m}$, we have
$\hat{s}/q \to \infty$; in particular $\hat{s} \ge q$ for all
sufficiently large $n$, so the construction $S_2(\hat{s},q)$ is
well-defined. The sequence $S_2(\hat{s}, q)$ has:
\begin{itemize}
\item alphabet size $q^2 \leq m$,
\item order $\leq \hat{s} + 2q - 2 = s$,
\item length $\geq q^2 \cdot \bigl[\binom{q}{2}(\hat{s} - q) + q\bigr] - O(q^2)$.
\end{itemize}
By monotonicity of $\lambda$ in both parameters (enlarging the
alphabet or the allowed order cannot decrease the maximum length),
$\lambda(s,m) \ge |S_2|$.  The length satisfies
\[
|S_2| \;\ge\; q^2 \binom{q}{2}(\hat{s}-q)
= \frac{q^4}{2}\,s \;-\; O(q^5).
\]
Since $q = \sqrt{m}\,(1-o(1))$, we have $q^4 = m^2(1-o(1))$
and $q^5 = O(m^{5/2})$, hence
\[
|S_2| \;\ge\; \frac{m^2}{2}\,s \;-\; O(m^{5/2})
= \binom{m}{2}\,s\,(1-o(1)).
\]
\end{proof}

\section{Corollaries}

\begin{corollary}\label{cor:diagonal}
$\displaystyle\lim_{n \to \infty} \frac{\lambda(n,n)}{n^3} = \frac{1}{2}.$
\end{corollary}

\begin{proof}
Apply Theorem~\ref{thm:general} with $s = m = n$. 
\end{proof}

\begin{corollary}\label{cor:ratios}
For any constants $a, b > 0$, $\lambda(an, bn) \;\sim\; \frac{ab^2}{2}\,n^3.$
\end{corollary}

\begin{proof}
Set $s = an$, $m = bn$ in Theorem~\ref{thm:general}. 
\end{proof}

\begin{corollary}\label{cor:poly}
Fix $\alpha > \frac{1}{2}$.  Then $\lambda(n^\alpha, n) \;\sim\; \frac{n^{2+\alpha}}{2}$.
\end{corollary}

\begin{proof}
Set $s = n^\alpha$, $m = n$.  Then $s/\!\sqrt{m} = n^{\alpha - 1/2} \to \infty$.
\end{proof}

\begin{remark}
Corollary~\ref{cor:diagonal} improves the Roselle--Stanton lower bound $\lambda(n,n) \geq \frac{n^3}{3} - O(n^2)$ to $\frac{n^3}{2} - O(n^{5/2})$, matching the pigeonhole upper bound in the leading term.  Wellman and Pettie~\cite[Section~6]{WP16} state that ``the true leading constant of $\lambda(n,n)$ is only known approximately; it is in the interval $[\frac{1}{3}, \frac{1}{2}]$.''  Corollary~\ref{cor:diagonal} resolves this to~$\frac{1}{2}$.
\end{remark}

\section*{Acknowledgments}
Claude 4.6 and GPT 5.2 were used for proof development, exposition, and revision.


\begin{thebibliography}{9}

\bibitem{ASS89}
P.~Agarwal, M.~Sharir, and P.~Shor,
\emph{Sharp upper and lower bounds on the length of general Davenport--Schinzel sequences},
J.~Combin.\ Theory Ser.~A~\textbf{52} (1989), 228--274.

\bibitem{K02}
M.~Klazar,
\emph{Generalized Davenport--Schinzel sequences: results, problems, and applications},
Integers~\textbf{2} (2002), A11.

\bibitem{N10}
G.~Nivasch,
\emph{Improved bounds and new techniques for Davenport--Schinzel sequences and their generalizations},
J.~ACM~\textbf{57}(3) (2010).

\bibitem{P15}
S.~Pettie,
\emph{Sharp bounds on Davenport--Schinzel sequences of every order},
J.~ACM~\textbf{62}(5):36 (2015).

\bibitem{RS71}
D.\,P.~Roselle and R.\,G.~Stanton,
\emph{Some properties of Davenport--Schinzel sequences},
Acta Arithmetica~\textbf{XVII} (1971), 355--362.

\bibitem{WP16}
J.~Wellman and S.~Pettie,
\emph{Lower bounds on Davenport--Schinzel sequences via rectangular Zarankiewicz matrices},
Discrete Math.~\textbf{341} (2018), 1987--1993.

\end{thebibliography}
\end{document}